\documentclass[11pt,a4paper,twoside]{article}
\usepackage{amsthm, amsfonts,amsmath}

\newtheorem{theorem}{Theorem}
\newtheorem{corollary}{Corollary}
\newtheorem{definition}{Definition}
\newtheorem{example}{Example}
\newtheorem{lemma}{Lemma}
\newtheorem{proposition}{Proposition}
\newtheorem{remark}{Remark}

\def\Ric{\operatorname{Ric}}

\title{$\eta$-Ricci solitons and $\eta$-Einstein metrics \\ on weak $\beta$-Kenmotsu $f$-manifolds}

\author{Vladimir Rovenski
\footnote{Department of Mathematics, University of Haifa, Mount Carmel, 3498838 Haifa, Israel
\newline e-mail: {\tt vrovenski@univ.haifa.ac.il}}
}

\begin{document}

\date{}

\maketitle

\begin{abstract}
Recent interest among geometers in $f$-structures of K.~Yano is due to the study of topology and dynamics of contact foliations, which
generalize the flow of the Reeb vector field on contact manifolds to higher dimensions.
Weak metric structures introduced by V.~Rovenski and R.~Wolak as a generalization of Hermitian and K\"{a}hler structures, as well as $f$-structures,
allow a fresh look at the classical theory. In~this paper, we study a new $f$-structure of this kind, called the weak $\beta$-Kenmotsu $f$-structure,
as a generalization of K.~Kenmotsu's concept.
We~prove that a weak $\beta$-Kenmotsu $f$-manifold is locally a twisted product of the Euclidean space
and a weak K\"{a}hler mani\-fold. Our main results show that such manifolds with $\beta=const$ and equipped with an $\eta$-Ricci soliton structure
whose potential vector field satisfies certain conditions are $\eta$-Einstein~manifolds of constant scalar~curvature.


\end{abstract}



\section{Introduction}

Contact Riemannian geometry is of growing interest of mathematicians because of its importance for physics, e.g.,~\cite{Blair-survey}.
A metric $f$-structure on a smooth manifold $M^{2n+s}$, see~\cite{b1970,yano-1961}, being a higher dimensional analog of a contact metric structure
($s=1$), is defined by a skew-symmetric (1,1)-tensor $f$ of rank $2n$,
orthonormal vector fields $\{\xi_i\}_{1\le i\le s}$
and their dual 1-forms $\{\eta^i\}$
such~that
\begin{align*}
 & {f}^2 = -{\rm id} + \sum\nolimits_{\,i}{\eta}^i\otimes {\xi}_i,\quad {\eta}^i({\xi}_j)=\delta^i_j, \\
 & g({f} X,{f} Y)= g(X, Y) -\sum\nolimits_{\,i}{\eta^i}(X)\,{\eta^i}(Y),
 \quad X,Y\in\mathfrak{X}_M .
\end{align*}
A metric $f$-structure is a~special case of an {almost pro\-duct structure}
(with Naveira's 36 distinguished classes, see~\cite{nav-1983}),
defined by
complementary orthogonal distributions ${\cal D}=f(TM)$ and ${\cal D}^\perp=\ker f$.
For an $f$-contact manifold, the~$s$-dimensional distribution ${\cal D}^\perp$ is tangent to a totally geodesic foliation,
%
which is spanned by Killing vector fields $\{\xi_i\}$.
Such so-called $\mathfrak{g}$-foliation is defined by a homomorphism of an $s$-dimensional Lie algebra $\mathfrak{g}$ to the Lie algebra of all vector fields on $M$,
for example,~\cite{rst-137}.
Recent interest of geometers in $f$-structures is motivated by the study of the topology and dynamics of contact foliations, especially the existence of closed leaves.
Contact foliations gene\-ralize to higher dimensions the flow of
the Reeb vector field on contact manifolds, and ${K}$-structures are particular case of
contact structures, see \cite{Fil-2024}.
A~special class of
$f$-manifolds, known as Kenmotsu $f$-manifolds, see~\cite{BA-2019,FP06,FP07,SV-2016} (Kenmotsu manifolds when $s=1$, see \cite{kenmotsu1972class}),
can be characterized in terms of warped products of $\mathbb{R}^s$ and K\"{a}hler manifolds.

\smallskip

In \cite{RWo_2}, we defined metric structures
that generalize
Hermitian, in particular K\"{a}hler, structure, as well as metric $f$-structure, in particular ${\cal C}$-, ${\cal S}$-, K- and $f$-K- contact, structures. These so-called ``weak" $f$-structures (i.e., the complex structure is replaced by a nonsingular skew-symmetric tensor)
are useful for studying totally geodesic and totally umbilical foliations, Killing fields and Einstein-type metrics,
and allow a fresh look at the classical theory.
A weak metric $f$-structure
is a~special case of an {almost product structure}, defined by complementary orthogonal distributions ${\cal D}=f(TM)$ and ${\cal D}^\perp=\ker f$
of
$(M^{2n+s},g)$.
Foliations appear when one or both distributions are involutive.
Weak metric $f$-manifolds form a broad class:
a~warped product
of a Euclidean space $\mathbb{R}^s$ and a weak K\"{a}hler manifold (Definition~\ref{D-wK2} and Example~\ref{Ex-sHK}) is a weak $\beta$-Kenmotsu $f$-manifold (Definition~\ref{D-wK}),
and the
product of $\mathbb{R}^s$ and a weak K\"{a}hler manifold is a weak ${\cal C}$-manifold.

\smallskip

Solutions of some non-linear  partial differential equations decompose at large time $t$ into solitary waves running with constant speed -- the so-called solitons. Ricci solitons (\textbf{RS}) $\frac12\pounds_V\,g+{\rm Ric}=\lambda\,g$ (where $\pounds$ is the Lie derivative, $V$ a vector field
and $\lambda$ a real constant), being self-similar solutions of the Ricci flow $\partial g/\partial t = -2\Ric_g$, generalize (when~$\pounds_V\,g=0$) Einstein metrics $\Ric = \lambda\,g$, see~\cite{CLN-2006}.

The Ricci flow and RS were studied for Hermitian and K\"{a}hler, manifolds, as well as for almost contact metric, in particular Sasakian, manifolds.
Since some compact manifolds (and $f$-K-contact manifolds) don't admit Einstein metrics.
Cho-Kimura \cite{cho2009ricci} generalized the notion of RS 
to $\eta$-\textit{RS}:
\begin{align}\label{1.1}
 ({1}/{2})\,\pounds_{V}\,g +{\rm Ric} = \lambda\,g +\mu\,\eta \otimes\eta\quad\mbox{for some}\ \ \lambda,\mu\in\mathbb{R},
\end{align}
where $\eta$ is a 1-form on~$M$. If $V$ is a Killing vector field, then \eqref{1.1} reduces to an $\eta$-\textit{Einstein metric},
which is defined by
\begin{align}\label{Eq-2.10s}
 {\rm Ric}= a\,g + b\,\eta\otimes\eta\quad\mbox{for some}\ \ a,b\in C^\infty(M).
\end{align}
The questions arise, see~\cite{Blair-survey,cho2009ricci}: How do RS interact with weak $f$-structures? Does a weak metric $f$-manifold
equipped with RS carry Einstein-type~metrics?
In~\cite{rov-127}, $\eta$-RS and $\eta$-Einstein metrics for weak $f$-$K$-contact manifolds were studied.
In this paper, we introduce
weak $\beta$-Kenmotsu $f$-manifolds ($\beta f$-\textbf{KM}), see Definition~\ref{D-wK},
as a generalization of K.~Kenmotsu's concept, and explore their properties and geometrical interpretations.
We study
when a weak $\beta f$-KM equipped (it cannot be an Einstein manifold) with an $\eta$-RS structure \eqref{Eq-1.1} carries an $\eta$-Einstein metric \eqref{Eq-2.10} of constant scalar curvature.

\smallskip

The paper consists of an introduction and four sections. In~Section~\ref{sec:01}, we review the basics of the weak metric $f$-structure.
In Sections~\ref{sec:02-f-beta} and \ref{sec:03-beta-R}, we introduce weak $\beta f$-KM ({weak $f$-KM} when $\beta\equiv1$,
and {weak ${\mathcal C}$-manifolds} when $\beta\equiv0$),
derive their fundamental properties (Theorem~\ref{T-2.0}), give their geometrical interpretation in terms of the twisted structure (Theorem~\ref{T-2.1}),
and prove that a weak $\beta f$-KM with $\xi$-parallel Ricci tensor is an $\eta$-Einstein manifold (Theorem~\ref{Th-01}).
In~Section~\ref{sec:03-f-beta}, we study the interaction of weak $\beta f$-KM with $\eta$-RS.
We prove that additional $\eta$-Einstein structure ensures the constancy of the scalar curvature (Theorem~\ref{thm3.1A}).
We then show that if a weak $\beta f$-KM with $\beta=const$ is an $\eta$-RS whose non-zero potential vector field, either a contact vector field (Theorem~\ref{thm3.3}) or collinear with $\sum_i\xi_i$ (Theorem~\ref{thm3.4}), then it is an $\eta$-Einstein manifold.
The~results generalize some theorems in \cite{RP-1,rov-126} where $s=1$ and can be extended to the case $\beta\in C^\infty(M)$.

\section{Preliminaries: weak metric $f$-manifolds}
\label{sec:01}

In this section, we review the basics of a weak metric $f$-structure
as a higher dimensional analog of the weak almost contact metric structure, see \cite{RWo_2,rst-43,rst-137} and Section~5.3.8 in \cite{Rov-Wa-2021}.

First, we generalize the notion of framed $f$-structure, see \cite{b1970,gy-1970,yano-1961}, called $f.pk$-structure in \cite{FP07}.

\begin{definition}[see \cite{RWo_2,rst-137}]\label{D-basic}\rm
A~\textit{weak metric $f$-structure} on a smooth manifold $M^{2n+s}\ (n>0,\,s>1)$ is a set $({f},Q,{\xi_i},{\eta^i},g)$, where
${f}$ is a skew-symmetric $(1,1)$-tensor of rank $2\,n$, $Q$ is a self-adjoint nonsingular $(1,1)$-tensor,
${\xi_i}\ (1\le i\le s)$ are orthonormal
vector fields, ${\eta^i}$ are dual 1-forms, and $g$ is a Riemannian metric on $M$, satisfying
\begin{align}\label{2.1}
 {f}^2 = -Q + \sum\nolimits_{\,i}{\eta^i}\otimes {\xi_i},\quad {\eta^i}({\xi_j})=\delta^i_j,\quad
 Q\,{\xi_i} = {\xi_i}, \\
\label{2.2}
 g({f} X,{f} Y)= g(X,Q\,Y) -\sum\nolimits_{\,i}{\eta^i}(X)\,{\eta^i}(Y),\quad X,Y\in\mathfrak{X}_M,
\end{align}
and $M^{2n+s}({f},Q,{\xi_i},{\eta^i},g)$ is called a \textit{weak metric $f$-manifold}.
\end{definition}

Assume that the distribution ${\cal D}:=\bigcap_{\,i}\ker{\eta^i}$ is ${f}$-invariant,
thus ${\cal D}=f(TM)$, $\dim{\cal D}=2\,n$ and
\[
 {f}\,{\xi_i}=0,\quad {\eta^i}\circ{f}=0,\quad \eta^i\circ Q=\eta^i,\quad [Q,\,{f}]=0 .
\]
By the above,
the distribution ${\cal D}^\bot=\ker f$ is spanned by $\{\xi_1,\ldots,\xi_s\}$ and is invariant for $Q$.
Define a (1,1)-tensor $\tilde Q$ by
 $Q = {\rm id} + \tilde{Q}$,
and note that $[\tilde{Q},{f}]=0$ and $\eta^i\circ\widetilde Q=0$. We also obtain ${f}^3+{f} = -\tilde{Q}\,{f}$.

Putting $Y={\xi_i}$ in \eqref{2.2},
we get ${\eta^i}(X)=g(X,{\xi_i})$;
thus, each ${\xi_i}$ is orthogonal to ${\cal D}$.
Therefore, $TM$ splits as complementary orthogonal sum of its subbundles ${\cal D}$ and ${\cal D}^\bot$ -- an {almost product structure}.

A~distribution ${\cal D}^\bot\subset TM$ (integrable or not) is said to be {totally geodesic} if
its second fundamental form vanishes: $\nabla_X Y+\nabla_Y X\in{\cal D}^\bot$ for any vector fields $X,Y\in{\cal D}^\bot$ --
this is the case when {any geodesic of $M$ that is tangent to ${\cal D}^\bot$ at one point is tangent to ${\cal D}^\bot$ at all its points},
e.g., \cite{Rov-Wa-2021}.
By Frobenius theorem, any involutive distribution is integrable, i.e., it is tangent to the leaves of the foliation.
Any integrable and totally geodesic distribution determines a totally geodesic~foliation.

A weak metric $f$-structure $({f},Q,{\xi_i},{\eta^i})$ is said to be {\it normal} if the following tensor is zero:
\begin{align*}
 {\cal N}^{\,(1)}(X,Y) = [{f},{f}](X,Y) + 2\sum\nolimits_{\,i}d{\eta^i}(X,Y)\,{\xi_i},\quad X,Y\in\mathfrak{X}_M ,
\end{align*}
where the Nijenhuis torsion
of a (1,1)-tensor ${S}$ and the exterior derivative of a 1-form ${\omega}$ are given~by
\begin{align*}
\notag
 & [{S},{S}](X,Y) = {S}^2 [X,Y] + [{S} X, {S} Y] - {S}[{S} X,Y] - {S}[X,{S} Y],\quad X,Y\in\mathfrak{X}_M, \\
 & d\omega(X,Y) = (1/2)\,\{X({\omega}(Y)) - Y({\omega}(X)) - {\omega}([X,Y])\},\quad X,Y\in\mathfrak{X}_M.
\end{align*}
Using the Levi-Civita connection $\nabla$ of $g$, one can rewrite $[S,S]$ as
\begin{align}\label{4.NN}
 [{S},{S}](X,Y) = ({S}\nabla_Y{S} - \nabla_{{S} Y}{S}) X - ({S}\nabla_X{S} - \nabla_{{S} X}{S}) Y .
\end{align}
The following tensors, see \cite{rst-43, rov-127}:
${\cal N}^{\,(2)}_i, {\cal N}^{\,(3)}_i$ and ${\cal N}^{\,(4)}_{ij}$,
are well known for metric $f$-manifolds,~\cite{b1970}:
\begin{align*}
 {\cal N}^{\,(2)}_i(X,Y) &= (\pounds_{{f} X}\,{\eta^i})(Y) - (\pounds_{{f} Y}\,{\eta^i})(X)
 =2\,d{\eta^i}({f} X,Y) - 2\,d{\eta^i}({f} Y,X) ,  \\
 {\cal N}^{\,(3)}_i(X) &= (\pounds_{{\xi_i}}{f})X
 = [{\xi_i}, {f} X] - {f} [{\xi_i}, X],\\
 {\cal N}^{\,(4)}_{ij}(X) &= (\pounds_{{\xi_i}}\,{\eta^j})(X)
 = {\xi_i}({\eta^j}(X)) - {\eta^j}([{\xi_i}, X])
 = 2\,d{\eta^j}({\xi_i}, X) .
\end{align*}

\begin{remark}
\rm
Let $M^{2n+s}(f,Q,\xi_i,\eta^i,g)$ be a
weak metric $f$-manifold.
Consider the pro\-duct manifold $\bar M = M^{2n+s}\times\mathbb{R}^s$,
where $\mathbb{R}^s$ is a Euclidean space with a basis $\partial_1,\ldots,\partial_s$.
Define tensor $J$ on $\bar M$ by
 $J(X, \sum\nolimits_{\,i}a^i\partial_i) = (fX - \sum\nolimits_{\,i}a^i\xi_i, \sum\nolimits_{\,j}\eta^j(X)\partial_j)$,
where $a_i\in C^\infty(M)$.
Tensors ${\cal N}^{\,(1)}, {\cal N}^{\,(2)}_i, {\cal N}^{\,(3)}_i, {\cal N}^{\,(4)}_{ij}$ appear from the
 condition $[J, J]=0$, when we express the normality condition ${\cal N}^{\,(1)}=0$.
\end{remark}

\begin{proposition}[see \cite{rst-43}]
The normality condition ${\cal N}^{\,(1)}=0$ for a weak metric $f$-structure implies
\begin{align}\label{Eq-normal}
 & {\cal N}^{\,(3)}_i = {\cal N}^{\,(4)}_{ij} = 0,\quad {\cal N}^{\,(2)}_i(X,Y) = \eta^i([\widetilde QX, fY]), \\
\label{Eq-normal-2}
 & \nabla_{\xi_i}\,\xi_j\in{\cal D},\quad  [X,\xi_i]\in{\cal D}\quad
 (X\in{\cal D}).
\end{align}
In this case,
${\cal D}^\bot$ is a totally geodesic distribution.
\end{proposition}

The {fundamental $2$-form} $\Phi$ on $M^{2n+s}({f},Q,\xi_i,\eta^i,g)$ is defined by
 $\Phi(X,Y)=g(X,{f} Y)$ for all $X,Y\in\mathfrak{X}_M$.
Recall the co-boundary formula for exterior derivative $d$ on a $2$-form $\Phi$,
\begin{align}\label{E-3.3}
 3\,d\Phi(X,Y,Z) &=
 X\,\Phi(Y,Z) + Y\,\Phi(Z,X) + Z\,\Phi(X,Y) \notag\\
 &-\Phi([X,Y],Z) - \Phi([Z,X],Y) - \Phi([Y,Z],X) .
\end{align}

\begin{proposition}[see \cite{rst-43}]
For a weak metric $f$-structure
we get
\begin{align}\label{3.1-new}
\notag
 & 2\,g((\nabla_{X}{f})Y,Z) = 3\,d\Phi(X,{f} Y,{f} Z) - 3\, d\Phi(X,Y,Z) + g({\cal N}^{\,(1)}(Y,Z),{f} X)\notag\\
 & +\sum\nolimits_{\,i}\big({\cal N}^{\,(2)}_i(Y,Z)\,\eta^i(X) + 2\,d\eta^i({f} Y,X)\,\eta^i(Z) - 2\,d\eta^i({f} Z,X)\,\eta^i(Y)\big)
 + {\cal N}^{\,(5)}(X,Y,Z),
\end{align}
where ${\cal N}^{\,(5)}$ is the tensor field acting as
\begin{align*}
 & {\cal N}^{\,(5)}(X,Y,Z) = {f} Z\,(g(X, \widetilde QY)) - {f} Y\,(g(X, \widetilde QZ)) \\
 & + g([X, {f} Z], \widetilde QY) - g([X,{f} Y], \widetilde QZ) + g([Y,{f} Z] -[Z, {f} Y] - {f}[Y,Z],\ \widetilde Q X).
\end{align*}
\end{proposition}


Let $\Ric^\sharp$ be a~(1,1)-tensor adjoint to the {Ricci tensor}
-- the suitable trace of the curvature tensor:
\begin{equation*}
 {\rm Ric}\,(X,Y) = {\rm trace}_{\,g}(Z\to R_{\,Z,X}\,Y),\quad
 R_{X, Y} = [\nabla_{X},\nabla_{Y}] - \nabla_{[X,Y]}\quad (X,Y,Z\in\mathfrak{X}_M).
\end{equation*}
The scalar curvature of a Riemannian manifold $(M,g)$ is defined as ${r}={\rm trace}_{\,g}\Ric={\rm trace}\Ric^\sharp$.

The following formulas are well known, e.g., \cite[Eqs. (6) and (7)]{ghosh2019ricci}:
\begin{align}
\label{ff}
 & (\pounds_V\nabla)(X,Y) = \nabla_X\nabla_Y V - \nabla_{\nabla_X Y}V + R_{V,X} Y ,\\
\label{3.21}
 & (\pounds_{V}\nabla)(X, Y) = \pounds_{V}(\nabla_{X} Y) - \nabla_{X}\,(\pounds_{V} Y) - \nabla_{\pounds_{V} X}Y
 , \\
\label{Eq-6}
 & (\nabla_X\pounds_V\,g)(Y,Z) = g((\pounds_V\nabla)(X,Y), Z) + g((\pounds_V\nabla)(X,Z), Y), \\
\label{2.6}
 & (\pounds_{V}\nabla_{Z}\,g - \nabla_{Z}\,\pounds_{V}\,g - \nabla_{[V,Z]}\,g)(X,Y)
  = -g((\pounds_{V}\nabla)(Z,X),Y) -g((\pounds_{V}\nabla)(Z,Y),X), \\
\label{Eq-7}
 & (\pounds_V\,R)_{X,Y} Z = (\nabla_X\pounds_V\nabla)(Y,Z) - (\nabla_Y\pounds_V\nabla)(X,Z).
\end{align}

\begin{definition}[see \cite{rov-127}]
\rm
A weak metric $f$-manifold
(see Definition~\ref{D-basic})  is said to be \textit{$\eta$-Einstein},~if
\begin{align}\label{Eq-2.10}
  \Ric = a\,g - a\sum\nolimits_{\,i}\eta^i\otimes\eta^i + (a+b)\,\bar\eta\otimes\bar\eta\quad
 \mbox{for some}\ \ a,b\in C^\infty(M),
\end{align}
where $\bar\xi=\sum\nolimits_{\,i}\xi_i$ and $\bar\eta=\sum\nolimits_{\,i}\eta^i$.
An \textit{$\eta$-\textit{RS}} on a weak metric $f$-manifold $M^{2n+s}({f},Q,{\xi_i},{\eta^i},g)$ is defined~by
\begin{align}\label{Eq-1.1}
 (1/2)\,\pounds_V\,g + \Ric = \lambda\,g -\lambda\sum\nolimits_{\,i}\eta^i\otimes\eta^i +(\lambda+\mu)\bar\eta\otimes\bar\eta ,
\end{align}
where $V$ is a smooth vector field on $M$ and $\lambda,\mu$ are real constants.
\end{definition}

If $V$ is a Killing vector field, i.e., $\pounds_V\,g=0$, then \eqref{Eq-1.1} reduces to \eqref{Eq-2.10} with $a=\lambda$ and $b=\mu$.
Taking the trace of \eqref{Eq-2.10}, gives the
scalar curvature $r = (2\,n + s)\,a + s\,b$.
For $s=1$ and $Q={\rm id}$, the definitions \eqref{Eq-1.1} and \eqref{Eq-2.10} are well-known for almost contact metric~manifolds:
\eqref{Eq-1.1}~reduces to an $\eta$-RS \eqref{1.1},
and \eqref{Eq-2.10} reduces to an $\eta$-Einstein metric \eqref{Eq-2.10s}.

\section{Geometry of weak $\beta f$-KM}
\label{sec:02-f-beta}

In the following definition, we generalize the notions of $\beta$-KM ($s=1$), $f$-KM ($\beta=1,\ s>1$), see \cite{FP06,FP07,ghosh2019ricci,SV-2016}, and weak $\beta$-KM ($s=1$), see \cite{rov-126}.

\begin{definition}\label{D-wK}
\rm
A
weak metric $f$-manifold $M^{2n+s}({f},Q,\xi_i,\eta^i,g)$
will be called a \textit{weak $\beta f$-KM} (a \textit{weak $f$-KM}) when $\beta\equiv1)$,
if
\begin{align}\label{2.3-f-beta}
 (\nabla_{X}\,{f})Y=\beta\{g({f} X, Y)\,\bar\xi -\bar\eta(Y){f} X\}\quad (X,Y\in\mathfrak{X}_M),
\end{align}
where $\beta\in C^\infty(M)$.
If $\beta\equiv0$, then \eqref{2.3-f-beta} defines a \textit{weak ${\mathcal C}$-manifold}.
\end{definition}

Note that $\bar\eta(\xi_i)=\eta^i(\bar\xi)=1$ and $\bar\eta(\bar\xi)=s$.
Taking $X=\xi_j$ in \eqref{2.3-f-beta} and using ${f}\,\xi_j=0$, we get $\nabla_{\xi_j}{f}=0$, which implies
$\nabla_{\xi_i}\,\xi_j\perp {\cal D}$. This and the 1st equality in \eqref{Eq-normal-2} give
\begin{align}\label{Eq-normal-3}
 \nabla_{\xi_i}\,\xi_j =0\quad  (1\le i,j\le s),
\end{align}
thus, ${\cal D}^\bot$ of a weak $\beta f$-KM is tangent to a foliation with flat totally geodesic~leaves.

\begin{lemma}\label{Lem-1}
For a weak $\beta f$-KM 
the following formulas are true:
\begin{align}\label{2.3b}
 & \nabla_{X}\,\xi_i = \beta\{X -\sum\nolimits_{\,j}\eta^j(X)\,\xi_j\}\quad (1\le i\le s,\quad X\in\mathfrak{X}_M),\\
 \label{2.3c}
 & (\nabla_{X}\,\eta^i)(Y) = \beta\{g(X,Y) -\sum\nolimits_{\,j}\eta^j(X)\,\eta^j(Y)\}\quad (1\le i\le s,\quad X,Y\in\mathfrak{X}_M) .
\end{align}
\end{lemma}

\begin{proof}
Taking $Y=\xi_i$ in \eqref{2.3-f-beta} and using $g({f} X, \xi_i)=0$ and $\bar\eta(\xi_i)=1$, we get ${f}(\nabla_X\,\xi_i -\beta X)=0$.
Since ${f}$ is non-degenerate on ${\cal D}$ and has rank $2\,n$, we get $\nabla_X\,\xi_i -\beta X = \sum\nolimits_{\,p}c^p\xi_p$.
The inner~product with $\xi_j$ gives $g(\nabla_X\,\xi_i,\xi_j) = \beta\,g(X,\xi_j) - c^j$.
Using \eqref{Eq-normal-2} and \eqref{Eq-normal-3}, we find $g(\nabla_X\,\xi_i,\xi_j)=g(\nabla_{\xi_i}X,\xi_j)=0$; hence, $c^j=\beta\,\eta^j(X)$.
This proves~\eqref{2.3b}. Using $(\nabla_{X}\,\eta^i)(Y) = g(\nabla_{X}\,\xi_i, Y)$ and \eqref{2.3b}, we get \eqref{2.3c}.
\end{proof}

The following result generalizes Theorem~3.4 in \cite{SV-2016}.

\begin{theorem}\label{T-2.0}
A weak metric $f$-manifold
is a weak $\beta f$-KM if and only if the following conditions are valid:
\begin{align}\label{Eq-almost-K}
 {\cal N}^{\,(1)} = 0,\quad
 d\eta^i = 0,\quad
 d\Phi = 2\,\beta\,\bar\eta\wedge\Phi,\quad
 {\cal N}^{\,(5)}(X,Y,Z)=2\,\beta\,\bar\eta(X) g(f Y, \widetilde Q Z).
\end{align}
\end{theorem}

\begin{proof}
Let \eqref{2.3-f-beta} be true.
Using \eqref{2.3b}, we obtain
\begin{align}\label{2.4A}
 (\nabla_X\,\eta^i)Y = X g(\xi_i, Y) -g(\xi_i, \nabla_X\,Y)
 = g(\nabla_{X}\,\xi_i, Y)
 = \beta\{g(X,Y) -\sum\nolimits_{\,j}\eta^j(X)\,\eta^j(Y)\}
\end{align}
for all $X,Y\in\mathfrak{X}_M$.
By \eqref{2.4A}, $(\nabla_X\,\eta^i)Y=(\nabla_Y\,\eta^i)X$ is true.
Thus, for $X,Y\in {\cal D}$ we obtain
\[
 0 = (\nabla_X\,\eta^i)Y-(\nabla_Y\,\eta^i)X
 = -\beta\,g([X,Y], \xi_i)
\]
that means integrability of the distribution ${\cal D}$, or equivalently, $d\eta^i(X,Y)=0$ for all $i=1,\ldots,s$ and $X,Y\in {\cal D}$.
By this and ${\cal N}^{\,(4)}_{ij}=0$, see \eqref{Eq-normal}, we find
$ d\eta^i = 0$.
Using \eqref{2.3-f-beta} and \eqref{E-3.3},
we get
\[
 3\,d\Phi(X,Y,Z) = 2\,\beta\{
  \bar\eta(X) g({f}Z, Y)
 +\bar\eta(Y) g({f}X, Z)
 +\bar\eta(Z) g({f}Y, X)
 \}.
\]
On the other hand, we have
\[
 3(\bar\eta\wedge\Phi)(X,Y,Z) =
  \bar\eta(X) g({f}Z, Y)
 +\bar\eta(Y) g({f}X, Z)
 +\bar\eta(Z) g({f}Y, X).
\]
Thus,
$d\Phi = 2\,\beta\,\bar\eta\wedge\Phi$ is valid.
By \eqref{4.NN} with $S=f$, and \eqref{2.3-f-beta}, we get $[{f},{f}]=0$; thus, using $d\eta^i=0$, yields ${\cal N}^{\,(1)} = 0$.
Finally, from \eqref{3.1-new}, using \eqref{2.1} and \eqref{2.2}, we~obtain
\begin{align*}
 & g((\nabla_{X}{f})Y,Z) - \frac12\,{\cal N}^{\,(5)}(X,Y,Z) = 3\,\beta\big\{ (\bar\eta\wedge\Phi)(X,fY,fZ) - (\bar\eta\wedge\Phi)(X,Y,Z) \big\} \\
 & = \beta\big\{ -\bar\eta(X) g(QZ, fY) + \bar\eta(X) g(Z, fY) - \bar\eta(Y) g(fX, Z) - \bar\eta(Z) g(X, fY)\big\} \\
 & = \beta\big\{ \bar\eta(Z) g(fX, Y)  - \bar\eta(Y) g(fX, Z) - \bar\eta(X) g(fY, \widetilde Q Z) \big\}.
\end{align*}
From this, using \eqref{2.3-f-beta}, we get ${\cal N}^{\,(5)}(X,Y,Z)=2\,\beta\,\bar\eta(X) g(f Y, \widetilde Q Z)$.

Conversely, using  \eqref{2.1} and \eqref{Eq-almost-K} in \eqref{3.1-new}, we obtain
\begin{align*}
 & 2\,g((\nabla_{X}{f})Y,Z) = 6\,\beta\,(\bar\eta\wedge\Phi)(X,{f} Y,{f} Z) - 6\,\beta\,(\bar\eta\wedge\Phi)(X,Y,Z)
  +2\,\beta\,\bar\eta(X) g(\widetilde Q f Y, Z) \\
 & = 2\,\beta\,\big\{ -\bar\eta(X) g(fY, QZ) - \bar\eta(X) g({f}Z, Y) -\bar\eta(Y) g({f}X, Z) -\bar\eta(Z) g({f}Y, X)
  +\bar\eta(X) g(f Y, \widetilde QZ) \big\} \\
 & = 2\,\beta\{g({f} X, Y)\,g(\bar\xi, Z) -\bar\eta(Y) g({f}X, Z)\}  ,
\end{align*}
thus \eqref{2.3-f-beta} is true.
\end{proof}

\begin{definition}[\cite{rov-126}]\label{D-wK2}
\rm
An even-dimensional Riemannian manifold $(\bar M, \bar g)$ equipped with a skew-symmet\-ric {\rm (1,1)}-tensor $J$
(other than a complex structure) is called a \textit{weak Hermitian manifold} if
$J^{\,2}$ is negative definite.
If $\bar\nabla J=0$, where $\bar\nabla$ is the Levi-Civita connection of $\bar g$, then
$(\bar M,\bar g, J)$ is called a \textit{weak K\"{a}hler~manifold}.
\end{definition}

\begin{remark}\rm
L.P. Eisenhart \cite{E-1923} proved that if a Riemannian manifold $(\bar M, \bar g)$ admits a parallel symmetric 2-tensor other than the constant
multiple of $\bar g$, then it is reducible.
Several authors studi\-ed the (skew-)symmetric parallel 2-tensors and classified them, for example, \cite{Gupta-2020,H-2022}.
\end{remark}

Let $(\bar M,\bar g)$ be a Riemannian manifold.
A \textit{twisted product} $\mathbb{R}^s\times_\sigma\bar M$
is the product $M=\mathbb{R}^s\times\bar M$ with the metric $g=dt^2\oplus \sigma^2\,\bar g$,
where $\sigma>0$ is a smooth function on $M$.
Set $\xi_i=\partial_{\,t_i}$.
The~Levi-Civita connections, $\nabla$ of $g$ and $\bar\nabla$ of $\bar g$, are related as follows:

\noindent\ \
(i) $\nabla_{\xi_i}\,\xi_j= 0$,
$\nabla_X\,\xi_i=\nabla_{\xi_i}X = \xi_i(\log\sigma)X$ for $X\perp Span\{\xi_1,\ldots,\xi_s\}$.

\noindent\ \
(ii) $\pi_{1*}(\nabla_XY) = -g(X,Y)\,\pi_{1*}(\nabla\log\sigma)$,
where $\pi_1: M \to\mathbb{R}^s$ is the orthoprojector.

\noindent\ \
(iii) $\pi_{2*}(\nabla_XY)$ is the lift of $\bar\nabla_XY$, where $\pi_2: M \to\bar M$ is the orthoprojector.

\noindent
If $\sigma(t_1,\ldots,t_s)>0$ is a smooth function on $\mathbb{R}^s$, then we get a \textit{warped product} $\mathbb{R}^s\times_\sigma\bar M$.

\begin{theorem}\label{T-2.1}
A weak $\beta f$-KM 
is locally a twisted product $\mathbb{R}^s\times_\sigma\bar M$, where $\bar M(\bar g, J)$ is a weak Hermitian manifold
$($a warped product if $\nabla\beta\perp{\cal D}$, and then $\bar M(\bar g, J)$ a weak K\"{a}hler manifold$)$,
$\bar\xi(\log\sigma)=s\,\beta$ and $-\beta\,\bar\xi$ is the mean curvature vector of the distribution ${\cal D}$.
\end{theorem}

\begin{proof}
By \eqref{Eq-normal-3}, the distribution ${\cal D}^\bot$ is tangent to a foliation with flat totally geodesic~leaves,
and by the second equality of \eqref{Eq-normal-2}, the distribution ${\cal D}$ is tangent to a foliation.
By \eqref{2.3b}, the Weingarten operator $A_{\xi_i}=-(\nabla\,\xi_i)^\bot\ (1\le i\le s)$ on ${\cal D}$ is conformal: $A_{\xi_i}X = -\beta X\ (X\in{\cal D})$. Hence, ${\cal D}$ is tangent to a totally umbilical foliation with the mean curvature vector $H=-\beta\,\bar\xi$.
By \cite[Theorem~1]{pr-1993}, our manifold is locally a twisted product.
By the above property (ii) of a twisted product, $\bar\xi(\log\sigma)=s\,\beta$ is true.
If $X(\beta)=0\ (X\in{\cal D})$, then we get locally a warped product, see \cite[Proposition~3]{pr-1993}.
By \eqref{2.2}, the (1,1)-tensor $J=f|_{\,\cal D}$ is skew-symmetric and $J^2$ is negative definite.
To show $\bar\nabla J=0$, using \eqref{2.3-f-beta} we find
$(\bar\nabla_X J)Y=\pi_{2*}((\nabla_X{f})Y)=0$ for $X,Y\in{\cal D}$.
\end{proof}

\begin{example}\label{Ex-sHK}\rm
a) According to \cite{E-1923}, a weak K\"{a}hler~manifold with $J^2\ne c\,\bar g$, where $c\in\mathbb{R}$, is reducible.
Take two (or even more) Hermitian manifolds $(\bar M_i,\bar g_i, J_i)$, hence $J_i^2=-{\rm id}_{\,i}$.
The product $\prod_{\,i}(\bar M_i,\bar g_i, c_i J_i)$, where $c_i\ne0$ are different constants, is a weak
Hermitian manifold with $Q=\bigoplus_{\,i}c_i^2{\rm id}_{\,i}$.
Moreover, if $(\bar M_i,\bar g_i, J_i)$ are K\"{a}hler~manifolds, then $\prod_{\,i}(\bar M_i,\bar g_i, c_i J_i)$ is a weak K\"{a}hler~manifold.
b) Let $\bar M(\bar g,J)$ be a weak K\"{a}hler manifold and
$\sigma(t_1,\ldots,t_s)=c\,e^{\,\beta\sum t_i}$ a function on Euclidean space $\mathbb{R}^s$, where $c\ne0$ and $\beta$ are constants.
Then the warped product mani\-fold $M=\mathbb{R}^s\times_\sigma\bar M$ has a weak metric $f$-structure which satisfies \eqref{2.3-f-beta}.
Using \eqref{4.NN} with $S=J$, for a weak K\"{a}hler manifold, we get $[{J},{J}]=0$; hence, ${\cal N}^{\,(1)}=0$ is true.
\end{example}

\begin{corollary}
A weak $f$-KM $M^{2n+s}(f,Q,\xi_i,\eta^i,g)$ is locally a warped product $\mathbb{R}^s(t_1,\ldots,t_s)\times_\sigma \bar M$,
where $\sigma=c\,e^{\,\sum t_i}\ (c=const\ne0)$ and $\bar M(\bar g,J)$ is a weak K\"{a}hler manifold.
\end{corollary}

To simplify the calculations in the rest of the paper, we assume that $\beta=const\ne0$.

\section{Curvature of weak $\beta f$-KM}
\label{sec:03-beta-R}

In this section we study the curvature of weak $\beta f$-KM.

\begin{proposition}
For a weak $\beta f$-KM with $\beta=const$, we have
\begin{align}\label{2.4}
 & R_{X, Y}\,\xi_i = \beta^2\big\{\bar\eta(X)Y-\bar\eta(Y)X +\sum\nolimits_{\,j}\big(\bar\eta(Y)\eta^j(X) - \bar\eta(X)\eta^j(Y)\big)\xi_j\big\}
 \quad (X,Y\in\mathfrak{X}_M),\\
\label{2.5-f-beta}
 & \Ric^\sharp \xi_i = -2\,n\,\beta^2\bar\xi , \\
\label{3.1}
 & (\nabla_{\xi_i}\Ric^\sharp)X = - 2\,\beta\Ric^\sharp X
 - 4\,n\,\beta^3\big\{ s X - s\sum\nolimits_{\,j}\eta^j(X)\,\xi_j + \bar\eta(X)\,\bar\xi\,\big\}
 \quad (X\in\mathfrak{X}_M) ,\\
\label{3.1A-f-beta}
 & \xi_i(r) = -2\,\beta\{r + 2\,s\,n(2\,n + 1)\,\beta^2\}\quad (1\le i\le s), \\
\label{E-L-02a}
 &(\nabla_X\Ric^\sharp)\,\xi_i=
 - \beta\Ric^\sharp X - 2\,n\,\beta^3\big\{ s X - s\sum\nolimits_{\,j}\eta^j(X)\,\xi_j + \bar\eta(X)\,\bar\xi\,\big\}\quad (X\in\mathfrak{X}_M) .
\end{align}
\end{proposition}

\begin{proof} Taking covariant derivative of \eqref{2.3b} along $Y\in \mathfrak{X}_M$,
we get
\begin{align*}
 \nabla_Y\nabla_X\,\xi_i =
 -\beta^2\,\big\{ \big( g(X,Y) - \sum\nolimits_{\,q}\eta^q(Y)\eta^q(X)\big)\,\bar\xi + \bar\eta(X)\big(Y - \sum\nolimits_{\,p}\eta^p(Y)\xi_p\big)\big\}.
\end{align*}
Repeated application of \eqref{2.3b} and the foregoing equation in the {curvature~tensor} $R$ of the Riemannian manifold, we get \eqref{2.4}.
Using a local orthonormal basis $(e_q)$ of the ma\-nifold,
and the equality $\sum\nolimits_{\,p,q}\big(\bar\eta(Y)\eta^p(e_q) - \bar\eta(e_q)\eta^p(Y)\big) \eta^p(e_q)= (s-1)\,\bar\eta(Y)$,
we derive from \eqref{2.4}
\begin{align*}
 & g(\Ric^\sharp \xi_i,Y)= \sum\nolimits_{\,q} g(R_{e_q, Y}\,\xi_i, e_q) \\
 & = \beta^2\sum\nolimits_{\,q}\big\{\bar\eta(e_q)g(Y,e_q)-\bar\eta(Y)g(e_q,e_q)
 +\big(\bar\eta(Y)\eta^p(e_q) - \bar\eta(e_q)\eta^p(Y)\big) \eta^p(e_q)\big\} \\
 & = \beta^2 \big\{ g(Y, \bar\xi)-(2\,n+s)\bar\eta(Y)
 +(s-1)\,\bar\eta(Y)
 \big\}
 = -2\,n\,\beta^2 g(\bar\xi, Y),
\end{align*}
from which we get \eqref{2.5-f-beta}. Next, using \eqref{2.3b}, we get
\begin{align}\label{Eq-9}
 (\pounds_{\xi_i}\,g)(Y,Z) = g(\nabla_Y\,\xi_i, Z)+g(\nabla_Z\,\xi_i, Y) = 2\,\beta\big(g(Y,Z)-\sum\nolimits_{\,j}\eta^j(Y)\,\eta^j(Z)\big).
\end{align}
Taking covariant derivative of \eqref{Eq-9} along $X$ and using \eqref{2.3b} gives
\begin{align*}
  (\nabla_X\pounds_{\xi_i}\,g)(Y,Z) = 2\,\beta^2\big\{\sum\nolimits_{\,j}\eta^j(X)\big(\eta^j(Y)\,\bar\eta(Z) + \bar\eta(Y)\,\eta^j(Z)\big)
  -g(X,Y)\,\bar\eta(Z) -g(X,Z)\,\bar\eta(Y)\big\}
\end{align*}
for all $X,Y,Z\in\mathfrak{X}_M$.
Using this in \eqref{Eq-6} with $V=\xi_i$, we obtain
\begin{align}\label{Eq-9b}
\nonumber
 & g((\pounds_{\xi_i}\nabla)(X,Y), Z) + g((\pounds_{\xi_i}\nabla)(X,Z), Y) \\
 & = 2\,\beta^2\big\{\sum\nolimits_{\,j}\eta^j(X)\big(\eta^j(Y)\,\bar\eta(Z) + \bar\eta(Y)\,\eta^j(Z)\big)
  -g(X,Y)\,\bar\eta(Z) -g(X,Z)\,\bar\eta(Y)\big\}.
\end{align}
By a combinatorial computation, we find
\begin{align*}
 & g((\pounds_{\xi_i}\nabla)(Y,Z), X) + g((\pounds_{\xi_i}\nabla)(Y,X), Z) \\
 & = 2\,\beta^2\big\{\sum\nolimits_{\,j}\eta^j(Y)\big(\eta^j(Z)\,\bar\eta(X) + \bar\eta(Z)\,\eta^j(X)\big)
  -g(Y,Z)\,\bar\eta(X) -g(Y,X)\,\bar\eta(Z) \big\},\\
 & g((\pounds_{\xi_i}\nabla)(Z,X), Y) + g((\pounds_{\xi_i}\nabla)(Z,Y), X) \\
 & = 2\,\beta^2\big\{\sum\nolimits_{\,j}\eta^j(Z)\big(\eta^j(X)\,\bar\eta(Y) + \bar\eta(X)\,\eta^j(Y)\big)
  -g(Z,X)\,\bar\eta(Y) -g(Z,Y)\,\bar\eta(X) \big\}.
\end{align*}
Subtracting \eqref{Eq-9b} from the sum of the last two equations gives
\begin{align}\label{Eq-10}
 (\pounds_{\xi_i}\nabla)(Y,Z) = 2\,\beta^2\big\{\sum\nolimits_{\,j}\eta^j(Y)\,\eta^j(Z) - g(Y,Z)\big\}\,\bar\xi\qquad (Y,Z\in\mathfrak{X}_M).
\end{align}
Taking covariant derivative of \eqref{Eq-10} along $X$ and using \eqref{2.3b} gives
\begin{align*}
 & (\nabla_X\pounds_{\xi_i}\nabla)(Y,Z)
   = 2\,\beta^3\big\{ \big[ \big( g(X,Y) - \sum\nolimits_{\,j}\eta^j(X)\,\eta^j(Y)\big)\,\bar\eta(Z) \\
 & + \big( g(X,Z) - \sum\nolimits_{\,j}\eta^j(X)\,\eta^j(Z) \big)\,\bar\eta(Y) \big]\bar\xi
 +s\big(\sum\nolimits_{\,j}\eta^j(Y)\,\eta^j(Z) - g(Y,Z)\big)\,\big(X - \sum\nolimits_{\,p}\eta^p(X)\,\xi_p\big) \big\} .
\end{align*}
Using this in \eqref{Eq-7} with $V=\xi_i$, we obtain
\begin{align}\label{Eq-11}
\nonumber
 & (\pounds_{\xi_i}R)_{X,Y} Z = 2\,\beta^3 \big\{
 \big( g(X,Z) {-} \sum\nolimits_{\,j}\eta^j(X)\,\eta^j(Z) \big)\big(\bar\eta(Y)\,\bar\xi + s(Y-\sum\nolimits_{\,q}\eta^q(Y)\,\xi_q)\big) \\
 & - \big( g(Y,Z) - \sum\nolimits_{\,j}\eta^j(Y)\,\eta^j(Z) \big)\big(\bar\eta(X)\,\bar\xi + sX-s\sum\nolimits_{\,q}\eta^q(X)\,\xi_q\big)\big\}.
\end{align}
Contracting \eqref{Eq-11} over $X$, we deduce
\begin{align}\label{Eq-12}
 (\pounds_{\xi_i}\Ric)(Y,Z) = \sum\nolimits_{\,a}g((\pounds_{\xi_i}R)_{e_a,Y} Z, e_a)
 = -4\,s\,n\,\beta^3 \big( g(Y,Z)- \sum\nolimits_{\,j}\eta^j(Y)\,\eta^j(Z)\big).
\end{align}
Taking the Lie derivative of equality $\Ric(Y,Z)= g(\Ric^\sharp Y, Z)$, we obtain
\begin{align}\label{Eq-13}
 (\pounds_{\xi_i}\Ric)(Y,Z) = (\pounds_{\xi_i}\,g)(\Ric^\sharp Y,Z) +g((\pounds_{\xi_i}\Ric^\sharp)Y, Z).
\end{align}
On the other hand, replacing $Y$ by $\Ric^\sharp Y$ in \eqref{Eq-9} and using \eqref{2.5-f-beta}, we obtain
\begin{align}\label{Eq-14}
\nonumber
 (\pounds_{\xi_i}\,g)(\Ric^\sharp Y,Z)
 & = 2\,\beta\big(g(\Ric^\sharp Y, Z) - g(\Ric^\sharp \xi_j, Y)\,\eta^j(Z)\big) \\
 &= 2\,\beta\big(g(\Ric^\sharp Y, Z) +2\,n\beta^2\bar\eta(Y)\,\bar\eta(Z)\big) .
\end{align}
Applying \eqref{Eq-13} and \eqref{Eq-14} in \eqref{Eq-12} we get
\begin{align}\label{3.1b}
 (\pounds_{\xi_i}\Ric^\sharp)Y = -2\,\beta\Ric^\sharp Y - 4\,n\,\beta^3\big\{ sY - s\sum\nolimits_{\,j}\eta^j(Y)\,\xi_j + \bar\eta(Y)\bar\xi\,\big\} .
\end{align}
Using \eqref{2.3b}, we calculate
\begin{align*}
 & (\pounds_{\xi_i}\Ric^\sharp)Y = \pounds_{\xi_i}(\Ric^\sharp Y) -\Ric^\sharp\pounds_{\xi_i} Y  \\
 & =\nabla_{\xi_i}(\Ric^\sharp Y) -\nabla_{\Ric^\sharp Y}\,{\xi_i} -\Ric^\sharp\nabla_{\xi_i} Y +\Ric^\sharp\nabla_Y\,{\xi_i}
 = (\nabla_{\xi_i}\Ric^\sharp) Y.
\end{align*}
Using this in \eqref{3.1b}, gives \eqref{3.1}.
Contracting \eqref{3.1}, we get \eqref{3.1A-f-beta}.
Taking covariant derivative of \eqref{2.5-f-beta} along $X$ and using \eqref{2.3b} gives \eqref{E-L-02a}.
\end{proof}

Remark that by \eqref{2.5-f-beta} weak $\beta f$-KM with $s>1$ cannot be Einstein manifolds.

\begin{proposition}\label{prop3.1A}
For an $\eta$-Einstein \eqref{Eq-2.10} weak $\beta f$-KM 
with $\beta=const$, we obtain
\begin{align}\label{3.16}
 {\rm Ric} = \big(s\beta^2+\frac{r}{2\,n}\big)\,\big\{ g - \sum\nolimits_{\,j}\eta^j\otimes\eta^j\big\} -2\,n\beta^2 \bar\eta\otimes\bar\eta.
\end{align}
\end{proposition}

\begin{proof}
Tracing \eqref{Eq-2.10} gives $r=(2\,n+s)\,a+sb$.
Putting $X=Y=\xi_i$ in \eqref{Eq-2.10} and using \eqref{2.5-f-beta},
yields $a+b=-2\,n\,\beta^2$. Thus,
 $a = s\,\beta^2+\frac{r}{2\,n}$
 and
 $b = -(2\,n+s)\,\beta^2 - \frac{r}{2\,n}$,
and \eqref{Eq-2.10} yields
\eqref{3.16}.
\end{proof}

The following theorem generalizes \cite[Theorem~1]{ghosh2019ricci} with $\beta\equiv1$ and $Q={\rm id}$.

\begin{theorem}\label{Th-01}
Let a weak $\beta f$-KM $M^{2n+s}({f},Q,\xi_i,\eta^i,g)$  with $\beta=const$ satisfy $\nabla_{\bar\xi}\Ric^\sharp=0$.
Then $(M,g)$ is an $\eta$-Einstein manifold \eqref{Eq-2.10} of scalar curvature $r=-2\,s\,n(2\,n+1)\,\beta^2$.
\end{theorem}

\begin{proof}
By \eqref{3.1} and conditions,
\[
 \Ric^\sharp X = -2\,n\,\beta^2\big\{ sX - s\sum\nolimits_{\,j}\eta^j(X)\,\xi_j + \bar\eta(X)\bar\xi\,\big\},
\]
is valid.
Since \eqref{Eq-2.10} with $a=-2\,s\,n\,\beta^2$ and $b=2\,(s-1)\,n\,\beta^2$ is true,
$(M,g)$ is an $\eta$-Einstein manifold of constant scalar curvature $r=-2\,s\,n(2\,n+1)\,\beta^2$.
For $s=1$, $(M,g)$ is an Einstein manifold.
\end{proof}

\section{$\eta$-RS on weak $\beta f$-KM}
\label{sec:03-f-beta}

Here, we study the interaction of weak $\beta f$-KM with $\eta$-RS and generalize some results in~\cite{RP-1}.

First, we derive the following two lemmas.

The following result generalizes Lemma~4 in \cite{rov-126}.

\begin{lemma}\label{lem3.2}
Let $(g,V)$ represent an $\eta$-RS \eqref{Eq-1.1} on a weak $\beta f$-KM $M^{2n+s}({f},Q,\xi_i,\eta^i,g)$ with $\beta=const$.
Then
\begin{align}\label{3.7}
 & (\pounds_{V} \nabla)(X,\xi_i) = 2\,\beta\Ric^\sharp X + 4\,n\,\beta^3\big\{s X - s\sum\nolimits_{\,j}\eta^j(X)\,\xi_j + \bar\eta(X)\bar\xi\,\big\},\\
\label{3.9}
\notag
 &(\pounds_V R)_{X,Y}\,\xi_i = 2\,\beta\{(\nabla_X\Ric^\sharp)Y - (\nabla_Y\Ric^\sharp)X\}
 + 2\,\beta^2\{\bar\eta(X)\Ric^\sharp Y  - \bar\eta(Y)\Ric^\sharp X\} \\
 &\quad + 4\,n\beta^4\big\{\big[s Y {-} s\sum\nolimits_{\,j}\eta^j(Y)\xi_j +\bar\eta(Y)\,\bar\xi\,\big]\,\bar\eta(X)
 -[s X {-} s\sum\nolimits_{\,j}\eta^j(X)\xi_j+\bar\eta(X)\,\bar\xi\,]\,\bar\eta(Y) \big\}, \\
\label{3.9trace}
 & (\pounds_V R)_{X,\,\xi_j}\,\xi_i = 0\quad (X\in\mathfrak{X}_M,\ 1\le i,j\le s).
\end{align}
\end{lemma}

\begin{proof} Taking the covariant derivative of \eqref{Eq-1.1} along $Z\in\mathfrak{X}_M$ and using \eqref{2.3c}, we~get
\begin{align}\label{3.3A-f-beta}
\nonumber
 & \frac12\,(\nabla_Z\,\pounds_{V}\,g)(X,Y) = -(\nabla_{Z}\,{\rm Ric})(X,Y)
 +\beta[\mu+(\lambda+\mu)]\times\\
 & \times\big\{ \big(g(X, Z) - \sum\nolimits_{\,j}\eta^j(X)\eta^j(Z)\big)\,\bar\eta(Y)
   + \big(g(Y, Z) - \sum\nolimits_{\,j}\eta^j(Y)\eta^j(Z)\big)\,\bar\eta(X) \big\}
\end{align}
for all $X,Y\in\mathfrak{X}_M$.
Since Riemannian metric tensor is parallel, $\nabla g=0$, it follows from \eqref{2.6} that
\begin{align}\label{3.4}
 (\nabla_{Z}\,\pounds_{V}\,g)(X,Y) = g((\pounds_{V}\nabla)(Z,X),Y) + g((\pounds_{V}\nabla)(Z,Y),X).
\end{align}
Plugging \eqref{3.4} into (\ref{3.3A-f-beta}), we obtain
\begin{align}\label{3.5}
\nonumber
 & g((\pounds_{V}\nabla)(Z,X),Y) + g((\pounds_{V}\nabla)(Z,Y),X) = -2(\nabla_{Z}\,{\rm Ric})(X,Y) + \beta[\mu+(s-1)(\lambda+\mu)]\times\\
 & \times\big\{ \big(g(X, Z) - \sum\nolimits_{\,j}\eta^j(X)\eta^j(Z)\big)\,\bar\eta(Y)
   + \big(g(Y, Z) - \sum\nolimits_{\,j}\eta^j(Y)\eta^j(Z)\big)\,\bar\eta(X) \big\}
\end{align}
for all $X,Y,Z\in\mathfrak{X}_M$. Cyclically rearranging $X,Y$ and $Z$ in \eqref{3.5}, we obtain
\begin{align}\label{3.6}
\nonumber
 g((\pounds_{V}\nabla)(X,Y),Z) & = (\nabla_{Z}\,{\rm Ric})(X,Y) - (\nabla_{X}\,{\rm Ric})(Y,Z) - (\nabla_{Y}\,{\rm Ric})(Z,X)\\
 & +2\,\beta[\mu +(s-1)(\lambda+\mu)]\big(g(X, Y) {-} \sum\nolimits_{\,j}\eta^j(X)\eta^j(Y)\big)\,\bar\eta(Z).
\end{align}
Substituting $Y=\xi_i$ in \eqref{3.6} yields the following:
\begin{align*}
 g((\pounds_{V}\nabla)(X,\xi_i),Z) & = (\nabla_{Z}\,{\rm Ric})(X,\xi_i) - (\nabla_{X}\,{\rm Ric})(\xi_i,Z) - (\nabla_{\xi_i}\,{\rm Ric})(Z,X) .
\end{align*}
Applying \eqref{3.1} and \eqref{E-L-02a} to this, we obtain \eqref{3.7} for any $X\in\mathfrak{X}_M$.
Next, using (\ref{2.3b}) in the covariant derivative of (\ref{3.7}) along $Y\in\mathfrak{X}_M$,
and calculating $\nabla_Y\big(s\sum\nolimits_{\,j}\eta^j(X)\,\xi_j - \bar\eta(X)\bar\xi\,\big) = 0$, yields
\begin{align*}
 & (\nabla_Y(\pounds_V \nabla))(X, \xi_i) + \beta(\pounds_V \nabla)(X, Y) = 2\,\beta(\nabla_Y\Ric^\sharp)X
 + 2\,\beta^2\,\bar\eta(Y)\Ric^\sharp X \\
 & + 4\,n\,\bar\eta(Y)\beta^4\,\big\{s X - s\sum\nolimits_{\,j}\eta^j(X)\xi_j + \bar\eta(X)\,\bar\xi\,\big\}
\end{align*}
for any $X\in\mathfrak{X}_M$.
Plugging this in \eqref{Eq-7}
with $Z=\xi_i$, we obtain  \eqref{3.9} for all $X,Y\in\mathfrak{X}_M$.
 Substituting $Y=\xi_j$ in \eqref{3.9} gives
\begin{align}\label{3.9b}
\notag
 (\pounds_V R)_{X,\xi_j}\,\xi_i & = 2\,\beta\{(\nabla_X\Ric^\sharp)\xi_j - (\nabla_{\xi_j}\Ric^\sharp)X\}
 + 2\,\beta^2\{\bar\eta(X)\Ric^\sharp \xi_j  - \Ric^\sharp X\} \\
 & - 4\,s\,n\,\beta^4\,\big(X - \sum\nolimits_{\,j}\eta^j(X)\xi_j\big) .
\end{align}
Then using \eqref{2.5-f-beta}, \eqref{3.1} and \eqref{E-L-02a} in \eqref{3.9b}, yields \eqref{3.9trace}.
\end{proof}

The following lemma generalizes Lemma~6 in \cite{rov-126}.

\begin{lemma}\label{lem3.3}
Let $(g,V)$ represent an $\eta$-RS \eqref{Eq-1.1}, a weak $\beta f$-KM $M^{2n+s}({f},Q,\xi_i,\eta^i,g)$ with $\beta=const$.
Then $\lambda+\mu=-2\,n\,\beta^2$ is true.
\end{lemma}

\begin{proof}
Using \eqref{2.4} and $g(R_{X,\,\xi_i}Z, W)=-g(R_{Z,W}\,\xi_i,X)$, we derive
\begin{align}\label{E-3.19a}
 R_{X,\xi_i}Z = \beta^2\big\{ g(X,Z)\bar\xi - \bar\eta(Z) X + \sum\nolimits_{\,j}\eta^j(X)\big(\bar\eta(Z)\xi_j - \eta^j(Z)\,\bar\xi \big) \big\}.
\end{align}
Taking the Lie derivative along $V$ of
\[
 R_{X,\,\xi_j}\,\xi_i=\beta^2\big\{\sum\nolimits_{\,k}\eta^k(X)\xi_k - X\big\},
\]
see \eqref{2.4} with $Y=\xi_j$ (or \eqref{E-3.19a} with $Z=\xi_j$), and using \eqref{2.4} and \eqref{E-3.19a},
gives
\begin{align}\label{E-3.19ab}
\notag
 & (\pounds_V  R)_{X,\,\xi_j}\,\xi_i = - \beta^2\Big\{\bar\eta(X)\pounds_V\,\xi_j -\bar\eta(\pounds_V\,\xi_j)X
 + \sum\nolimits_{\,k}\{\bar\eta(\pounds_V\,\xi_j)\,\eta^k(X) -\bar\eta(X)\eta^k(\pounds_V\,\xi_j)\}\xi_k \Big\} \\
\notag
 & - \beta^2\Big\{g(X,\pounds_V\,\xi_i)\,\bar\xi
 -\bar\eta(\pounds_V\,\xi_i)X + \sum\nolimits_{\,k}\eta^k(X)\{\bar\eta(\pounds_V\,\xi_i)\,\xi_k -\eta^k(\pounds_V\,\xi_i)\,\bar\xi\}\Big\} \\
 & + \beta^2\sum\nolimits_{\,k}\big\{ (\pounds_V\,\eta^k)(X)\xi_k + \eta^k(X)\pounds_V\,\xi_k\big\}.
\end{align}
Here we used $\pounds_V  (R_{X,\,\xi_j}\,\xi_i) =(\pounds_V  R)_{X,\,\xi_j}\,\xi_i + R_{X,\pounds_V\,\xi_j}\,\xi_i + R_{X,\,\xi_j}\,\pounds_V\,\xi_i$ with
\begin{align*}
 & R_{X,\pounds_V\,\xi_j}\,\xi_i = \beta^2\big\{\bar\eta(X)\pounds_V\,\xi_j -\bar\eta(\pounds_V\,\xi_j)X
 + \sum\nolimits_{\,k}\{\bar\eta(\pounds_V\,\xi_j)\,\eta^k(X) -\bar\eta(X)\eta^k(\pounds_V\,\xi_j)\}\xi_k \big\},\\
 & R_{X,\,\xi_j}\,\pounds_V\,\xi_i =\beta^2\big\{g(X,\pounds_V\,\xi_i)\,\bar\xi
 -\bar\eta(\pounds_V\,\xi_i)X +\sum\nolimits_{\,k}\eta^k(X)\{\bar\eta(\pounds_V\,\xi_i)\,\xi_k -\eta^k(\pounds_V\,\xi_i)\,\bar\xi\}\big\}.
\end{align*}
In view of \eqref{3.9trace}, the equation \eqref{E-3.19ab} divided by $\beta^2$ becomes
\begin{align}\label{E-3.19}
\notag
 & \sum\nolimits_{\,k}(\pounds_V\,\eta^k)(X)\xi_k + \sum\nolimits_{\,k}\eta^k(X)\pounds_V\,\xi_k
 - \bar\eta(X)\pounds_V\,\xi_j + \bar\eta(\pounds_V\,\xi_j)X \\
\notag
 & - \bar\eta(\pounds_V\,\xi_j)\sum\nolimits_{\,k}\eta^k(X)\xi_k +\bar\eta(X)\sum\nolimits_{\,k}\eta^k(\pounds_V\,\xi_j)\xi_k
 - g(X,\pounds_V\,\xi_i)\,\bar\xi +\bar\eta(\pounds_V\,\xi_i)X  \\
 & - \bar\eta(\pounds_V\,\xi_i)\sum\nolimits_{\,k}\eta^k(X)\,\xi_k +\sum\nolimits_{\,k}\eta^k(X)\eta^k(\pounds_V\,\xi_i)\,\bar\xi
 = 0 .
\end{align}
For $X\in{\cal D}$
the equation \eqref{E-3.19} reduces to the following:
\begin{align}\label{E-3.19b}
 \sum\nolimits_{\,k}(\pounds_V\,\eta^k)(X)\xi_k  + \bar\eta(\pounds_V\,\xi_i)X + \bar\eta(\pounds_V\,\xi_j)X - g(X,\pounds_V\,\xi_i)\,\bar\xi  = 0 .
\end{align}
Taking the ${\cal D}$- and ${\cal D}^\bot$- components of \eqref{E-3.19b} yields
\begin{align}\label{E-3.19c}
 \bar\eta(\pounds_V\, \xi_i)=0 ,\quad
 (\pounds_V\,\eta^k)(X)  = g(X,\pounds_V\,\xi_i)\quad (X\in{\cal D}).
\end{align}
Using \eqref{2.5-f-beta}, we write \eqref{Eq-1.1} with $Y=\xi_k$ as
\begin{align}\label{E-3.20}
 (\pounds_V\,g)(X,\xi_k)
 = 2\,(2\,n\beta^2+\lambda+\mu)\,\bar\eta(X) .
\end{align}
Using the equality
\[
 (\pounds_V\,g)(\xi_i,\xi_k)
 = -g(\xi_i, \pounds_V\,\xi_k)-g(\xi_k, \pounds_V\,\xi_i)
 = -\eta^i(\pounds_V\,\xi_k)-\eta^k(\pounds_V\,\xi_i),
\]
the equation \eqref{E-3.20} for $X=\xi_i$ reduces to
\begin{align}\label{E-3.21}
 \eta^i(\pounds_V\,\xi_k) +\eta^k(\pounds_V\,\xi_i) = -2(2\,n\beta^2+\lambda +\mu) \ \Rightarrow\
  \bar\eta(\pounds_V\,\bar\xi) = -s^2(2\,n\beta^2+\lambda +\mu).
\end{align}
Comparing \eqref{E-3.21} with $\bar\eta(\pounds_V\,\bar\xi) = 0$,
see \eqref{E-3.19c}, we achieve the required result $\lambda+\mu=-2\,n\,\beta^2$.
\end{proof}


The following theorem generalizes Theorem~3 in \cite{rov-126} where $s=1$.

\begin{theorem}\label{thm3.1A}
Let $(g,V)$ represent an $\eta$-RS \eqref{Eq-1.1} on a weak $\eta$-Einstein \eqref{Eq-2.10} $\beta f$-KM
with $\beta=const$.
If~$s>1$, then we assume $V\in{\cal D}$.
Then $a=-2\,s\,n\beta^2$, $b=2(s-1)n\beta^2$ and the scalar curvature is equal to $r=-2\,s\,n(2\,n+1)\,\beta^2$;
moreover, if $s=1$ then the manifold is Einstein.
\end{theorem}

\begin{proof}
Set $s>1$.
Taking covariant derivative of \eqref{3.16} along $Y$ and using \eqref{2.3b}, we~get
\begin{align}\label{AA}
\nonumber
 & (\nabla_Y \Ric^\sharp)X = \frac{Y(r)}{2n}\big\{X-\sum\nolimits_{\,j}\eta^j(X)\xi_j\big\} \\
\nonumber
 & -\big((2\,n+s)\,\beta^2 + \frac{r}{2\,n} \big)\beta\big\{\,g(X,Y)\,\bar\xi + \bar\eta(X)\big( Y  - \sum\nolimits_{\,j}\eta^j(Y)\xi_j\big)
 -\sum\nolimits_{\,i}\eta^i(X)\,\eta^i(Y)\,\bar\xi\,\big\} \\
 & - 2(s-1)n\beta^3\big\{ \big(g(X,Y) -\sum\nolimits_{\,p}\eta^p(X)\,\eta^p(Y)\big)\,\bar\xi
 + \bar\eta(X)\big(Y - \sum\nolimits_{\,p}\eta^p(Y)\xi_p \big) \big\}.
\end{align}
Contracting \eqref{AA} over $Y$ and using the well known identity ${\rm div}_g\Ric = \frac12\,dr$,
we get
\begin{align}\label{AA-2}
 (n-1)\,X(r) = -\sum\nolimits_{\,i}\eta^i(X)\,\xi_i(r) - 2\,n\beta\big\{r + 2\,s\,n(2\,n + 1)\,\beta^2\big\}\,\bar\eta(X).
\end{align}
Using \eqref{3.1A-f-beta} in \eqref{AA-2} yields
\begin{align}\label{AD}
 X(r) = -2\,\beta\big\{r + 2\,s\,n(2\,n + 1)\,\beta^2\big\}\,\bar\eta(X) \quad  (X\in\mathfrak{X}_M) ;
\end{align}
hence $r$ is constant along the leaves of ${\cal D}$.
Using \eqref{3.16} and \eqref{AA} in \eqref{3.9}, and then applying \eqref{AD} and Lemma~\ref{lem3.2} gives
 $(\pounds_V R)_{X,Y}\,\xi_i = 0$
for all $X,Y\in \mathfrak{X}_M$.  Therefore,
we get
\begin{align}\label{AF}
 (\pounds_V \Ric)(Y, \xi_i) = {\rm trace}\{X \to (\pounds_V R)_{X,Y}\,\xi_i\}
 = 0.
\end{align}
Equation \eqref{2.5-f-beta} gives $\Ric(Y,\xi_i)=-2\,n\,\beta^2\,\bar\eta(Y)$. Taking its Lie derivative along $V$ yields
\[
 (\pounds_V \Ric)(Y, \xi_i) +\Ric(Y,\pounds_V\,\xi_i) = -2\,n\,\beta^2(\pounds_V\,\bar\eta)(Y)
\]
for all $Y\in\mathfrak{X}_M$.
Inserting \eqref{AF} in the preceding equation, we have
\begin{align}\label{Eq-A00}
\Ric(Y,\pounds_V\,\xi_i)= -2\,n\,\beta^2\big\{(\pounds_V\,g)(Y,\bar\xi) +g(Y,\pounds_V\,\bar\xi)\big\}
= -2\,n\,\beta^2\,g(Y,\pounds_V\,\bar\xi).
\end{align}
In view of \eqref{Eq-1.1}, \eqref{2.5-f-beta} and \eqref{3.16}, the equation \eqref{Eq-A00} becomes
\begin{align}\label{AI}
 \Big(s\beta^2 +\frac{r}{2\,n}\Big)\big\{ g(Y,\pounds_V\,\xi_i) -
 \sum\nolimits_{\,p}\eta^p(Y)\,\eta^p(\pounds_V\,\xi_i)\big\} 
  -2\,n\beta^2 \bar\eta(Y)\,\bar\eta(\pounds_V\,\xi_i) = -2\,n\,\beta^2 g(Y,\pounds_V\,\bar\xi).
\end{align}
For $Y\in{\cal D}$, \eqref{AI} reduces to the following:
\begin{align}\label{AJ-1}
 \big(s\beta^2+\frac{r}{2\,n}\big)\,g(Y,\pounds_V\,\xi_i) = -2\,n\,\beta^2 g(Y,\pounds_V\,\bar\xi)\quad (Y\in{\cal D}),
\end{align}
from which we obtain
\begin{align}\label{AJ}
 \big(2\,s\,n(2\,n+1)\beta^2 + r\big)g(Y,\pounds_V\,\bar\xi)=0\quad (Y\in{\cal D}).
\end{align}

\noindent
\textbf{Case I}. Let's assume that $(M,g)$ has constant scalar curvature $r=-2\,s\,n(2\,n+1)\,\beta^2$.
Then by \eqref{3.16}, we obtain
\[
 {\rm Ric}^\sharp X = -2\,s\,n\beta^2 \big\{ X - \sum\nolimits_{\,j}\eta^j(X)\,\xi_j\big\}  -2\,n\,\beta^2 \bar\eta(X)\,\bar\xi .
\]
Hence, $(M,g)$ is an $\eta$-Einstein manifold \eqref{Eq-2.10} with $a=-2\,s\,n\beta^2$ and $b=2(s-1)n\beta^2$.

\noindent
\textbf{Case II}. Let's assume that $s(2\,n+1)\beta^2+\frac{r}{2\,n}\ne 0$ on an open set $\mathcal{U}$ of $M$.
Then $\pounds_V\,\bar\xi=[V,\bar\xi]=0$ on $\mathcal{U}$, see \eqref{AJ} and \eqref{Eq-normal-2}.
Let us show that this leads to a contradiction.
If $\pounds_V\,\xi_i\ne0$ for some $i$, then from \eqref{AJ-1} and $\pounds_V\,\xi_i\in{\cal D}$, see \eqref{Eq-normal-2}, we get $s\beta^2+\frac{r}{2\,n}=0$.
Using the previous equality in \eqref{3.16}, we get $\Ric^\sharp X = 0$ for all $X\in{\cal D}$.
By this and \eqref{2.5-f-beta}, the following is true:
\[
 \Ric^\sharp X = -2\,n\,\beta^2\bar\eta(X)\,\bar\xi.
\]
Therefore, using Lemma~\ref{Lem-1}, we obtain
\[
 (\nabla_{\xi_i}\Ric^\sharp)X
 = -2\,n\beta^2\nabla_{\xi_i}\big(\bar\eta(X)\bar\xi\big) - \Ric^\sharp(\nabla_{\xi_i}X)
 = 0\quad (1\le i\le s).
\]
By the previous equality, $(\nabla_{\bar\xi}\Ric^\sharp)X = 0$ is true, hence by Theorem~\ref{Th-01},
we get $r=-2\,s\,n(2\,n+1)\,\beta^2$ -- a contradiction.
Therefore, $\pounds_V\,\xi_i=[V,\xi_i]=0$ for all $i$ on some open set $\mathcal{V}\subset \mathcal{U}$.
It~follows that
\begin{align}\label{AK}
 \nabla_{\xi_i}V=\nabla_V\,\xi_i = \beta\{V-\sum\nolimits_{\,p}\eta^p(V)\,\xi_p\},
\end{align}
where we have used \eqref{2.3b}.
Replacing $Y$ by $\xi_i$ in \eqref{ff} and using \eqref{2.3b}, \eqref{2.4} and \eqref{AK}, we get
\begin{align}\label{LV-Nabla}
 (\pounds_V\nabla)(X,\xi_i) = -\beta^2\{g(X,V) - \sum\nolimits_{\,j}\eta^j(X)\eta^j(V)\}\bar\xi.
\end{align}
Further, from \eqref{3.7} and \eqref{LV-Nabla}, we get
\begin{align}\label{b-Ric}
 \Ric^\sharp X = - 2\,n\,\beta^2\big\{ s\big(X - \sum\nolimits_{\,j}\eta^j(X)\xi_j\big) + \bar\eta(X)\bar\xi\,\big\}
 -(\beta/2)\{g(X,V) - \sum\nolimits_{\,j}\eta^j(X)\eta^j(V)\}\bar\xi.
\end{align}
Comparing ${\cal D}$-components of \eqref{b-Ric} and \eqref{3.16},
yields a contradiction: $r=-2\,s\,n(2\,n+1)\beta^2$ on ${\cal V}$.
\end{proof}


\begin{definition}
\rm
A vector field $X$ on a weak metric $f$-manifold $M({f},Q,\xi_i,\eta^i,g)$
is called a {\it contact vector field},
 if
the there exists a function $\rho\in C^\infty(M)$ such~that
\begin{align}\label{3.20}
 \pounds_{X}\,\eta^i=\rho\,\eta^i\quad(1\le i\le s),
\end{align}
and if $\rho=0$, i.e., the flow of $X$ preserves $\eta^i$, then $X$ is said to be a \textit{strictly contact vector field}.
\end{definition}

We study the interaction of a weak $\beta f$-KM with an $\eta$-RS whose potential vector field $V$ is~either a contact vector field or collinear to $\bar\xi$.
The following theorem generalizes Theorem~4 ($s=1$) in~\cite{rov-126}.

\begin{theorem}\label{thm3.3}
Let $M^{2n+s}({f},Q,\xi_i,\eta^i,g)$, be a weak $\beta f$-KM with $\beta=const$.
If $(g,V)$ represents an $\eta$-RS \eqref{Eq-1.1}, with a contact potential vector field $V$, then $V$ is strictly contact
and the manifold is $\eta$-Einstein \eqref{Eq-2.10} with $a=-2\,s\,n\beta^2,\ b=2(s-1)n\beta^2$ of scalar curvature $r=-2\,s\,n(2\,n+1)\beta^2$.
\end{theorem}

\begin{proof}
Taking Lie derivative of $\eta^i(X)=g(X,\xi_i)$ along $V$ and using \eqref{3.20} and $(\pounds_V\,g)(X,\xi_i)=0$, see \eqref{E-3.20}, we obtain $\pounds_V\,\xi_i=\rho\,\xi_i$.
Then, using $\pounds_V\,\xi_i\in{\cal D}$, see the second equality in \eqref{Eq-normal-2}, we get $\rho=0$.
Therefore, $\pounds_V\,\xi_i=0$ and $V$ is a strictly contact vector field. Also, \eqref{3.20} gives $\pounds_{V}\,\eta^j=0$.
Setting $Y=\xi_i$ in \eqref{3.21} and using \eqref{2.3b} and the equality
$(\pounds_{V}\,\eta^j)(X) = V(\eta^j(X)) - \eta^j([V, X])$,
we find
\begin{align}\label{3.22}
\nonumber
 & (\pounds_V \nabla)(X, \xi_i) = \beta\pounds_V(X-\sum\nolimits_{\,p}\eta^p(X)\xi_p)
 -\beta(\pounds_V X -\sum\nolimits_{\,p}\eta^p(\pounds_V X)\,\xi_p) \\
 & = -\beta\sum\nolimits_{\,p}\{(\pounds_V\,\eta^p)(X)\xi_p +\eta^p(X)\pounds_V\,\xi_p\} + \beta\sum\nolimits_{\,p}\eta^p(\pounds_V X)\,\xi_p .
\end{align}
From \eqref{3.22}, since $\pounds_V\,\eta^p=\pounds_V\,\xi_p=0$ is true and the distribution ${\cal D}$ is involutive, i.e., $\pounds_Y X\in{\cal D}\ (X,Y\in{\cal D})$, we obtain $(\pounds_V \nabla)(X, \xi_i)=0$. Using \eqref{3.7}, we get
\begin{align*}
 \Ric^\sharp X & = -2\,n\,\beta^2\big\{ sX - s\sum\nolimits_{\,j}\eta^j(X)\,\xi_j + \bar\eta(X)\bar\xi\,\big\} .
\end{align*}
Therefore, our $(M,g)$ is an $\eta$-Einstein manifold \eqref{Eq-2.10} with $a=-2\,s\,n\beta^2,\ b=2(s-1)n\beta^2$
and constant scalar curvature
$r=-2\,s\,n(2\,n+1)\beta^2$.
\end{proof}

The following theorem generalizes Theorem~5 in \cite{rov-126} (where $s=1$).

\begin{theorem}\label{thm3.4}
Let $M^{2n+s}({f},Q,\xi_i,\eta^i,g)$, be a weak $\beta f$-KM with $\beta=const$.
If~$(g,V)$ represents an $\eta$-RS \eqref{Eq-1.1} with a potential vector field $V$ collinear to $\bar\xi$:
$V=\delta\,\bar\xi$ for a smooth function $\delta\ne0$ on $M$, then $\delta=const$
and the manifold is $\eta$-Einstein \eqref{Eq-2.10} with $a=-2\,s\,n\beta^2$ and $b=2(s-1)n\beta^2$ of constant scalar curvature $r=-2\,s\,n(2\,n+1)\beta^2$.
\end{theorem}

\begin{proof} Using \eqref{2.3-f-beta} in the covariant derivative of $V=\delta\,\bar\xi$ with any $X\in\mathfrak{X}_M$ yields
\[
 \nabla_X V = X(\delta)\,\bar\xi +\delta\,\beta(X-\sum\nolimits_{\,j}\eta^j(X)\,\xi_j)\quad (X\in\mathfrak{X}_M).
\]
Using this and calculations
\begin{align*}
& (\pounds_{\delta\,\bar\xi}\,g)(X,Y)=\delta(\pounds_{\bar\xi}\,g)(X,Y) +X(\delta)\,\bar\eta(Y) +Y(\delta)\,\bar\eta(X),\\
& (\pounds_{\bar\xi}\,g)(X,Y)=2\,s\,\beta\{g(X,Y) -\sum\nolimits_{\,j}\eta^j(X)\,\eta^j(Y)\},
\end{align*}
we transform the $\eta$-RS equation \eqref{Eq-1.1} into
\begin{align}\label{3.23}
\notag
 2\,{\rm Ric}(X,Y) =& -X(\delta)\,\bar\eta(Y) -Y(\delta)\,\bar\eta(X) 
 +2(\lambda-\delta\beta)\big\{ g(X,Y) - \sum\nolimits_{\,j}\eta^j(X)\,\eta^j(Y)\big\} \\
 & - 4\,n\beta^2\bar\eta(X)\,\bar\eta(Y)\quad (X,Y\in\mathfrak{X}_M).
\end{align}
Inserting $X=Y=\xi_i$ in \eqref{3.23} and using \eqref{2.5-f-beta} and $\lambda+\mu=-2\,n\,\beta^2$, see Lemma~\ref{lem3.3},
we get $\xi_i(\delta)=0$. It~follows from \eqref{3.23} and \eqref{2.5-f-beta}
that $X(\delta)=0\ (X\in{\cal D})$. Thus $\delta$ is constant on $M$, and \eqref{3.23} reads~as
\begin{align*}
 {\rm Ric}= (\lambda-\delta\beta)\,\big\{ g - \sum\nolimits_{\,j}\eta^j\otimes\eta^j\big\} -2\,n\beta^2\bar\eta\otimes\bar\eta.
\end{align*}
This shows that $(M,g)$ is an $\eta$-Einstein manifold with $a=\lambda-\delta\beta$ and $b=-\lambda+\delta\beta-2\,n\beta^2$ in \eqref{Eq-2.10}.
From Proposition~\ref{prop3.1A} we conclude that $\lambda=\delta\beta-2\,s\,n\beta^2$, $\mu=-\delta\beta+2(s-1)n\beta^2$,
and the scalar curvature of $(M,g)$ is $r=-2\,s\,n(2\,n+1)\beta^2$.
\end{proof}


\begin{thebibliography}{00}



\bibitem{BA-2019}
Balkan, Y.S., Aktan, N. Almost Kenmotsu $f$-manifolds. Carpathian Math. Publ. 2015, 7(1), 6--21


\bibitem{Blair-survey}
 Blair, D.E. {A survey of Riemannian contact geometry}, Complex Manifolds, {2019}, {6}, 31--64

\bibitem{b1970}
 Blair, D.\,E.
 Geometry of manifolds with structural group $U(n)\times O(s)$,  {J. Diff. Geom.} {4} (1970), 155--167



\bibitem{cho2009ricci}
 Cho, J., Kimura, M. Ricci soliton and real hypersurfaces in a complex space form, Tohoku Math. J. 61(2), 205--212 (2009)

\bibitem{CLN-2006}
 Chow, B.,  Lu, P., Ni, L. {\em Hamilton's Ricci flow}. Grad. Stud. in Math., 77, AMS,
 RI,
 2006

\bibitem{E-1923}
Eisenhart, L.P. Symmetric tensors of the second order whose first covariant derivates are zero.
Trans. Amer. Math. Soc., 25 (2023), 297--306




\bibitem{FP06}
Falcitelli, M., Pastore, A.M. $f$-Structures of Kenmotsu type, Mediterr. J. Math. 3 (2006), 549--564

\bibitem{FP07}
Falcitelli, M., Pastore, A.M. Almost Kenmotsu $f$-manifolds, Balkan J. Geom. Appl., 12 (1) (2007) 32--43

\bibitem{Fil-2024}
Finamore, D. Contact foliations and generalised Weinstein conjectures. {Ann. Global Anal. Geom}. {65}\,(4), Paper No. 27, 32 p. ({2024})

\bibitem{ghosh2019ricci}
 Ghosh, A. Ricci soliton and Ricci almost soliton within the framework of Kenmotsu manifold, Carpathian Math. Publ., 11(1), 59--69, (2019)




\bibitem{gy-1970}
Goldberg, S.\,I., Yano, K. On normal globally framed $f$-manifolds, Tohoku Math. J. 22 (1970), 362--370

\bibitem{Gupta-2020}
Gupta, P., Singh, S.K. Second order parallel tensor on generalized $f$.pk-space form and hypersurfaces of ge\-neralized $f$.pk-space form.
Differential Geometry - Dynamical Systems, Vol. 23, 2021, pp. 59--66. Balkan Society of Geometers, Geometry Balkan Press 2021

\bibitem{H-2022}
Herrera, A.C. {Parallel skew-symmetric tensors on 4-dimensional metric Lie algebras}.
Revista de la Uni\'{o}n Matem\'{a}tica Argentina, {65}, no. 2 (2023), 295--311

\bibitem{kenmotsu1972class}
Kenmotsu, K. A class of almost contact Riemannian manifolds, T\^{o}hoku Math. J., 24 (1972), 93--103

\bibitem{nav-1983}
Naveira, A. A classification of Riemannian almost product manifolds, {Rend. Math.}, {3} ({1983}), 577--592


\bibitem{RP-1}
 Patra, D.S., Rovenski, V. Almost $\eta$-Ricci solitons on Kenmotsu manifolds, European J. of Mathematics, 7 (2021), 1753--1766
	

\bibitem{rov-126}
 Patra, D.S., Rovenski, V.  Weak $\beta$-Kenmotsu manifolds and $\eta$-Ricci solitons, pp. 53--72.
 In: Rovenski, V., Walczak, P., Wolak, R. (eds) \textit{Differential Geometric Structures and Applications},
2024. Springer Proceedings in Mathematics and Statistics, 440.
Springer, Cham 

\bibitem{pr-1993}
Ponge, R., Reckziegel, H. {Twisted products in pseudo-Riemannian geometry}, Geom. Dedicata,  {48} (1993), 15--25

\bibitem{Rov-Wa-2021}
 Rovenski, V. and Walczak, P.G. \textit{Extrinsic geometry of foliations}. Progress in Mathematics, vol.~339, Birkh\"{a}user, 2021

\bibitem{RWo_2}
 Rovenski, V., Wolak, R.
{New metric structures on $\mathfrak{g}$-foliations}, Indagationes Math.
33 (2022), 518--532

\bibitem{rst-43}
 Rovenski, V. Metric structures that admit totally geodesic foliations, J. Geom. (2023) 114:32.



\bibitem{rov-127}
 Rovenski, V. Einstein-type metrics and Ricci-type solitons on weak $f$-K-contact manifolds, pp. 29--51.
In: Rovenski, V., Walczak, P., Wolak, R. (eds) \textit{Differential Geometric Structures and Applications},
2024. Springer Proceedings in Mathematics and Statistics, 440. Springer, Cham 

\bibitem{rst-137}
Rovenski, V. Geometry of weak metric $f$-manifolds: a survey. Mathematics 2025, 13(4), 556

\bibitem{SV-2016}
Sari, R., Turgut Vanli, A. Generalized Kenmotsu manifolds, Communications in Mathematics and Applications, 7, No. 4, 311--328, 2016



\bibitem{yano-1961}
Yano, K. On a structure $f$ satisfying $f^3+f=0$, Tech.
Report No. 12, Univ. Washington, 1961

\end{thebibliography}
\end{document}